\documentclass[12pt]{article}
\usepackage[margin=1in]{geometry}
\usepackage{graphicx}
\usepackage{amsfonts}
\usepackage{amsmath}
\usepackage{amsthm}
\usepackage{amssymb}

\newcommand \trace{\mathrm{tr}}
\newcommand \ord{\mathrm{ord}}
\newcommand \weylD{D}

\theoremstyle{definition} \newtheorem{defi}{Definition}[section]
\theoremstyle{plain} \newtheorem{prop}{Proposition}[section]
\theoremstyle{plain} \newtheorem{thm}{Theorem}[section]
\theoremstyle{plain} \newtheorem{coro}{Corollary}[section]
\theoremstyle{definition} \newtheorem{example}{\emph{Example}}[section]

\def\mydate{\leavevmode\hbox{\the\month/\the\day/\the\year}}

\author{Robert Ream}
\title{The Adjunction Inequality for Weyl-Harmonic Maps}
\date{\mydate}

\begin{document}

\maketitle

\begin{abstract}
In this paper we study an analog of minimal surfaces called Weyl-minimal surfaces in conformal manifolds with a Weyl connection $(M^4,c,\weylD)$.
We show that there is an Eells-Salamon type correspondence between nonvertical $\mathcal{J}$-holomorphic curves in the weightless twistor space and branched Weyl-minimal surfaces. When $(M,c,J)$ is conformally almost-Hermitian, there is a canonical Weyl connection.  We show that for the canonical Weyl connection, branched Weyl-minimal surfaces satisfy the adjunction inequality
\begin{equation*} \chi(T_f\Sigma)+\chi(N_f\Sigma) \le \pm c_1(f^*T^{(1,0)}M). \end{equation*}
The $\pm J$-holomorphic curves are automatically Weyl-minimal and satisfy the corresponding equality.
These results generalize results of Eells-Salamon and Webster for minimal surfaces in K\"ahler 4-manifolds as well as their
extension to almost-K\"ahler 4-manifolds by Chen-Tian, Ville, and Ma.
\end{abstract}

\begin{section}{Introduction}
This paper describes an extension of the notion of minimal surfaces to the setting of conformal manifolds with a Weyl connection.
Particular attention is given to almost-Hermitian 4-manifolds endowed with their canonical Weyl connection. 
We first review the relevant standard theory.

According to Eells and Salamon \cite{eells1985}, branched minimal surfaces in an oriented Riemannian 4-manifold $M$ have a one-to-one correspondence with
nonvertical $\mathcal{J}$-holomorphic curves in the twistor space $\mathcal{Z}$, where $\mathcal{J}$ is the canonical non-integrable almost-complex structure on $\mathcal{Z}$. 
Applying twistor techniques, they further show that if $M$ is almost-K\"ahler with almost-complex structure $J$, the $\pm J$-holomorphic curves are minimal.  When $M$ is K\"ahler they prove the adjunction inequality,
\[ \chi(T_f\Sigma)+\chi(N_f\Sigma) \le \pm c_1(f^*T^{(1,0)}M), \]
where $T_f\Sigma$ is the tangent bundle to $\Sigma$ ramified at the branch points of $f$, and $N_f\Sigma$ is its normal bundle in $f^*TM$.
Concurrently, Webster \cite{webster1984} obtained his formulas (\ref{web1}), (\ref{web2})
 for a minimal surface in a K\"ahler 4-manifold, which imply the adjunction inequality.
The adjunction inequality was extended to minimal surfaces in almost-K\"ahler 4-manifolds by Chen-Tian \cite{chen1997}, Ville \cite{ville1997}, and Ma \cite{ma1998}.

This leads to the following picture for almost-K\"ahler manifolds:
The adjunction inequality holds for minimal surfaces;
every $\pm J$-holomorphic curve is minimal, and equality holds in (\ref{adj}) with the corresponding sign. 

For an almost-Hermitian manifold, in general, the $\pm J$-holomorphic curves are not minimal, and
in \cite{chen1997} they remark that the (\ref{adj}) will not hold for minimal surfaces.
In this paper we show that the above scenario for almost-K\"ahler manifolds can be extended to almost-Hermitian manifolds 
when considering a conformally invariant condition on surfaces related to the minimal condition. 
We now briefly describe this condition and list our main theorems.

Let $M$ be a manifold with conformal metric $c$ and Weyl connection $\nabla^D$, to be described in detail later.
For $i:\Sigma\to M$ an immersed submanifold, the Weyl second fundamental form $B$ is defined in \cite{Pedersen1995} as follows.
Taking $g\in c$, there is a one-form $\alpha_g$ such that $\nabla^D g=-2\alpha_g\otimes g$.
Let $A_g$ be the usual second fundamental form, then
\[ B = A_g - \left(\alpha_g^{\sharp_g}\right)^{\perp}\otimes g. \]
We say the submanifold is \emph{Weyl-minimal} if $\trace_{i^*g}B=0$. 
Branched Weyl-minimal then has the obvious meaning. 

We extend the Eells-Salamon twistor correspondence as follows.
\begin{thm}
Let $(M^4,c,\weylD)$ be a Weyl manifold, and $(\Sigma,[\eta])$ a Riemann surface.  There are complex structures $\mathcal J_\pm$ on the weightless twistor spaces $\mathcal{Z}_\pm$ which give a 1-to-1 correspondence between non-vertical $\mathcal J_\pm$-holomorphic curves $\tilde{f}_\pm:\Sigma\to\mathcal{Z}_\pm$ and non-constant weakly conformal branched Weyl-minimal immersions $f:\Sigma\to M$.
\end{thm}
If the Weyl derivative is exact, then the Weyl-minimal surfaces are minimal for a preferred metric in $c$, and this is just the usual correspondence for that metric.
\begin{thm}\label{JisHOL}
Let $(M^4,c,J,D)$ be a conformally almost-Hermitian manifold with its canonical Weyl connection. The almost-complex structure gives rise to a $\mathcal J_+$-holomorphic section of $\mathcal{Z}_+$.
\end{thm}
This and the previous theorem imply the following corollary.
\begin{coro}
Under the assumptions of Theorem $\ref{JisHOL}$, a $\pm J$-holomorphic curve $f:\Sigma\to M$ is a weakly conformal branched Weyl-minimal immersion.
\end{coro}
Finally we prove that Webster's formulas hold for branched Weyl-minimal immersions.
\begin{thm}\label{webforthm}
For a Riemann surface $(\Sigma,[\eta])$ and a conformally almost-Hermitian manifold with its canonical Weyl-connection $(M^4,c,J,D)$, if $f:\Sigma \to M$ is a weakly conformal branched Weyl-minimal immersion with $P$ complex points and $Q$ anti-complex points then
\begin{align}
\label{web1} \chi(T_f\Sigma)+\chi(N_f\Sigma) &= -P-Q \\
\label{web2} c_1(f^*T^{(1,0)}M) &= P-Q.
\end{align}
\end{thm}
The adjunction inequality follows from $P$ and $Q$ being positive.
\begin{coro}\label{adjcor}
For a Riemann surface $(\Sigma,[\eta])$ and a conformally almost-Hermitian manifold with its canonical Weyl-connection $(M^4,c,J,D)$, if $f:\Sigma \to M$ is a weakly conformal branched Weyl-minimal immersion then
\begin{equation}\label{adj} \chi(T_f\Sigma)+\chi(N_f\Sigma) \le \pm c_1(f^*T^{(1,0)}M). \end{equation}
\end{coro}
The corresponding equality holds for $\pm J$-holomorphic curves.
\end{section}

\begin{section}{Preliminaries}
\begin{subsection}{Weyl Geometry}
By definition, the \emph{density bundle} on an $n$-dimensional manifold $M$ is $L:=\left|\Lambda^nTM\right|^{\frac{1}{n}}$.
The tensor bundles $L^w\otimes TM^j \otimes T^*M^k$ are said to have \emph{weight} $w+j-k$.
A \emph{Weyl derivative} $D$ is a connection on the density bundle.
A \emph{conformal metric} $c$ is a metric on the weightless tangent bundle $L^{-1}TM$
satisfying the normalizing condition $|\!\det c|=1$.
This can also be considered as a metric on $TM$ with values in $L^2$.
\begin{defi}
The triple $(M,c,D)$ is called a \emph{Weyl Manifold}.
\end{defi}
The bundle $L$ is trivial, and a nowhere zero section of $L$, $\mu$, is called a \emph{length scale}.
This defines a metric in the conformal class $c$ by $g_\mu = \mu^{-2} c$.  
The section $\mu$ gives a trivialization of $L$ which has a corresponding trivializing connection $D^\mu$.
This defines a one form $\alpha_\mu = D - D^\mu$, so that 
\[ D(h\mu)=(dh + h D) \mu=(dh+h(\alpha_\mu+D^\mu))\mu=(dh+h\alpha_\mu)\mu.\]
There is a unique torsion free connection $\nabla^D$ on $TM$ making $c$ parallel,
\[ \nabla^D_X Y = \nabla^{g_\mu}_X Y + \alpha_\mu(X)Y+\alpha_\mu(Y)X-g_\mu(X,Y)\alpha_\mu^{\sharp_{g_\mu}}, \]
where $\nabla^{g_\mu}$ is the Levi-Civita connection for the metric $g_\mu$.
\begin{subsubsection}{Weightless Twistor Space}
When $M$ is oriented, $c$ defines a section $\nu_c$ of the orientation bundle $L^n\Lambda^n T^*M$.
This can be used to define the conformal Hodge star
\[\star:L^m \Lambda^k T^*M \to L^{m+n-2k} \Lambda^{n-k} T^*M, \]
where for $\beta,\gamma\in L^k \Lambda^k T^*M$
\[ \beta\wedge\star\gamma = c(\beta,\gamma)\nu_c \]
For $n=4$ and $m=0$, $\star:\Lambda^2 T^*M\to\Lambda^2 T^*M$ is an involution with $\pm 1$ eigenspaces $\Lambda^2_\pm T^*M$.
The weightless twistor spaces \cite{calderbank1999} can be constructed as the sphere bundles 
\[ \mathcal{Z}_\pm = S(L^{2}\Lambda^2_\pm T^*M). \]

We now review the construction of an almost-complex structure $\mathcal{J}_\pm$ on $\mathcal{Z}_\pm$.
This can be seen by working at a point $q_\pm\in\mathcal{Z}_\pm$ which projects to $p\in M$.
For $U$ a neighborhood of $p$, and a local section $s_\pm:U\to\mathcal{Z}_{\pm|U}$ satisfying $s_\pm(p)=q_\pm$, there is a weightless K\"ahler form $\sigma_\pm$ given by this section and a corresponding almost-complex structure $J_\pm$ on $T_pM$ given by
\[ \sigma_\pm(X,Y) = c(J_\pm X,Y). \]
As the fiber of $\mathcal{Z}_\pm$ at $p$ is a sphere in $L^2\Lambda^2_\pm T^*_p M$, the vertical tangent space at $q_\pm$ is the space perpendicular to $\sigma_\pm$ in $L^2\Lambda^2_\pm T^*_p M$.  This is the space of weightless $J_\pm$-anti-invariant 2-forms \cite{Draghici2013}.

\begin{defi}
The space of weightless $J_\pm$-anti-invariant 2-forms $L^2\Lambda^{2,J_\pm}_-T^*_pM$ is the $(-1)$-eigenspace for the involution
$I_\pm:L^2\Lambda^2 T^*_pM\to L^2\Lambda^2 T^*_pM$ given by $(I_\pm\beta)(X,Y)=\beta(J_\pm X,J_\pm Y)$.
\end{defi}
There is an induced almost-complex structure acting on $\beta\in L^2\Lambda^{2,J_\pm}_-T^*_pM$ by 
\[ (J_\pm \beta)(X,Y) = \beta(J_\pm X,Y). \]
To see that this is an almost-complex strucure, first note that $J_\pm\beta$ is a two form as
\[\beta(J_\pm X,Y)= -\beta(Y,J_\pm X) = \beta(J_\pm Y, J^2_\pm X) = -\beta(J_\pm Y, X).\]
Second, 
\[I_\pm J_\pm\beta=J_\pm I_\pm\beta=-J_\pm\beta,\]
so $J_\pm\beta\in L^2\Lambda^{2,J_\pm}_-T^*_pM$.
Finally, it is easily seen that $J_\pm^2\beta=-\beta$.

Extending $\beta$ to be complex bilinear gives
\begin{align*}
\frac{1}{4}\beta(X+iJ_\pm X,Y+iJ_\pm Y) &=\frac{1}{2}(\beta+i J_\pm\beta)(X,Y), \\
\frac{1}{4}\beta(X-iJ_\pm X,Y+iJ_\pm Y) &= 0 \\
\frac{1}{4}\beta(X-iJ_\pm X,Y-iJ_\pm Y) &= \frac{1}{2}(\beta-i J_\pm\beta)(X,Y).
\end{align*}
Therefore $\beta\in L^2(\Lambda^{2,0}_\pm T^*_pM\oplus\Lambda^{0,2}_\pm T^*_pM)$. This shows that $L^2\Lambda^{2,J_\pm}_-T^*_pM\perp \sigma_\pm$ as 
there is an orthogonal splitting
\[ \Lambda^2_\pm T^*_pM\otimes\mathbb{C}=\Lambda^{2,0}_\pm T^*_pM\oplus\Lambda^{0,2}_\pm T^*_pM\oplus\mathbb{C}\sigma_\pm. \]
Furthermore $\frac{1}{2}(\beta-i J_\pm\beta)\in L^2\Lambda^{2,0}_\pm T^*_pM$ and $\frac{1}{2}(\beta+i J_\pm\beta)\in L^2\Lambda^{0,2}_\pm T^*_pM$.

There is an isomorphism, $\beta\mapsto\beta^v$, from $L^2\Lambda^{2,J_\pm}_-T^*_pM$ to the vertical tangent space $V(T_{q_\pm}\mathcal{Z}_\pm)$, so that for $\beta\in L^2\Lambda^{2,J_\pm}_-T^*_pM$ we have $\beta^v\in V(T_{q_\pm}\mathcal{Z}_\pm)$.
This isomorphism and the connection induce an isomorphism, $X\mapsto X^h$, from $T_pM$ to the horizontal tangent space $H(T_{q_\pm}\mathcal{Z}_\pm)$ by
\begin{equation}\label{vhConn} ds_\pm(X) = X^h + (\nabla^D_X \sigma_\pm)^v. \end{equation}
The almost-complex structure on $\mathcal{Z}_\pm$ is now given by linearly extending
\begin{align}\label{calJh} \mathcal{J}_\pm(X^h) &:= (J_\pm X)^h, \\
\label{calJv}\mathcal{J}_\pm(\beta^v) &:= (J_\pm \beta)^v. 
\end{align}
In \cite{eells1985} Eells and Salamon used a similar complex structure to study weakly conformal harmonic maps.
Their complex structure is the same as the complex structure of Penrose, studied by Atiyah, Hitchen and Singer in \cite{AtiyahHitchenSinger}, except that it reverses the orientation of the fibers.
The complex structure defined in (\ref{calJh}) and (\ref{calJv}) differs from that of Eells and Salamon only in the use of a Weyl connection to define the horizontal space rather than the Levi-Civita connection.
\end{subsubsection}
\begin{subsubsection}{Submanifold Geometry}

Let $(M,c,D)$ be a Weyl manifold, and $i:\Sigma\to M$ an immersed submanifold. 
Then $\Sigma$ inherits a conformal structure $\bar{c}$ and Weyl derivative $\bar{D}$.
One way to see this is to choose a length scale, $\mu\in\Gamma(L)$. 
Then the metric $g_\mu$ and the one form $\alpha_\mu$ can be pulled back to $\Sigma$ as $\bar{g}_\mu=i^*g_\mu$ and $\bar{\alpha}_\mu=i^*\alpha_\mu$.
Hence $\bar{\mu}=|\!\det \bar{g}_\mu|^{-1/(2\dim\Sigma)}$ is a section of the density bundle of $\Sigma$.
The inherited conformal metric is $\bar{c}=\bar{\mu}^2\bar{g}_\mu$ and the inherited Weyl derivative is 
\[ \bar{D}(h\bar{\mu}) = (dh + h\bar{\alpha}_\mu) \bar{\mu}. \]
The connection $\bar{\nabla}^{\bar{D}}$ on $\Sigma$ is defined so that $\bar{c}$ is parallel,
\[ \bar{\nabla}^{\bar{D}}_X Y = \bar{\nabla}^{\bar{g}_\mu}_X Y + \bar{\alpha}_\mu(X)Y+\bar{\alpha}_\mu(Y)X-\langle X,Y\rangle_{\bar{g}_\mu}\bar{\alpha}_\mu^{\sharp_{\bar{g}_\mu}}. \]
The Weyl second fundamental form \cite{Pedersen1995} is given by
\[B^D(X,Y)=\nabla^D_X Y - \bar{\nabla}^{\bar{D}}_X Y.\]
Equivalently, for $A_{g_\mu} = \nabla^{g_\mu}-\bar{\nabla}^{\bar{g}_\mu}$
\[B^D(X,Y)=A_{g_\mu}(X,Y) - \langle X,Y\rangle_{g_\mu}\left(\alpha_\mu^{\sharp_{g_\mu}}\right)^\perp. \]
The Weyl mean curvature is 
\begin{equation}\label{weylH}\mathbf{H}^D = \frac{1}{\dim\Sigma}\trace_{\bar{g}_\mu}B^D = \mathbf{H}_{g_\mu}-\left(\alpha_\mu^{\sharp_{g_\mu}}\right)^\perp,
\end{equation}
where $\mathbf{H}_{g_\mu}$ is the usual mean curvature of $\Sigma$ with respect to the metric $g_\mu$.
\begin{defi}
The immersion $i:\Sigma\to M$ is \emph{Weyl-minimal} if $\mathbf{H}^D=0$.
\end{defi}
\begin{example}{Exact and Closed Weyl Derivatives}\\
If $(M,c,D)$ is a Weyl manifold and there is a length scale $\mu$ so that $\alpha_\mu$ is exact, then $D$ is called exact.
If $\alpha_\mu=du$, then $D=D^{e^{-u}\mu}$ and the Weyl-minimal surfaces are just the minimal surfaces for the metric $g^{e^{-u}\mu}=e^{2u}g^\mu$.
Similarly, if $\alpha_\mu$ is closed then $D$ is called closed.
In this case, if $f:\Sigma\to M$ is a Weyl-minimal branched immersion then there is a lift to the universal cover $\tilde f:\tilde\Sigma\to\tilde M$.
The conformal metric and Weyl Derivative can be lifted to $\tilde M$ and the closed Weyl derivative becomes exact.
Thus $\tilde f$ is a minimal surface for a metric in the lifted conformal class.
\end{example}
The harmonic map equation can also be generalized to this setting.
In \cite{kokarev2009} the second fundamental form of a map $f:\Sigma\to M$ is defined 
for manifolds $\Sigma$ and $M$ with torsion-free connections $\nabla^\Sigma$ and $\nabla^M$.
If $\nabla$ is the induced connection on $T^*\Sigma\otimes f^*TM$, then the second fundamental form is just $\nabla df$.
If $\eta$ is a metric on $\Sigma$ then the \textit{tension} of the map can be defined as
\[\tau(\eta,\nabla^\Sigma,\nabla^M) = \trace_\eta \nabla df. \]
A map is \emph{psuedo-harmonic} if the tension field is zero.
We study the case where the domain $(\Sigma,\eta)$ is a Riemannian manifold with its Levi-Civita connection $\nabla^\eta$
and the target manifold $(M,c,D)$ is a Weyl manifold.
This is opposite of the case studied in \cite{kokarev2009}, where the domain is Weyl and the target is Riemannian.

\begin{defi}
A map $f:\Sigma\to M$ is \emph{Weyl-harmonic} if $\tau(\eta,\nabla^\eta,\nabla^D)=0$.
\end{defi}
From this point we only consider the case where $\Sigma$ has dimension two.
Using local isothermal coordinates on $\Sigma$ so that $\eta=e^{2\lambda}(dx^2+dy^2)$, the tension field is
\begin{equation}\label{tension} \tau(\eta,\nabla^\eta,\nabla^D)=e^{-2\lambda}(\nabla^D_{\partial_x}f_x+\nabla^D_{\partial_y}f_y), \end{equation}
where $f_x=df(\partial_x)$.
In terms of the complex coordinate $z=x+iy$ this is just
\begin{equation}\label{tensionz} \tau(\eta,\nabla^\eta,\nabla^D)=e^{-2\lambda}\nabla^D_{\partial_{\bar z}}f_z. \end{equation}
This can also be written more explicitly as
\begin{equation}\label{tensionex} \tau(\eta,\nabla^\eta,\nabla^D)=
e^{-2\lambda}(\nabla^{g_\mu}_{\partial_x}f_x+\nabla^{g_\mu}_{\partial_y}f_y+2\alpha_\mu(f_x)f_x+2\alpha_\mu(f_y)f_y-(|f_x|^2_{g_\mu}+|f_y|^2_{g_\mu})\alpha_\mu^{\sharp_{g_\mu}}). \end{equation}
From this we see that when $\Sigma$ has dimension 2
\[ \tau(e^{2u}\eta,\nabla^{e^{2u}\eta},\nabla^D)=e^{-2u}\tau(\eta,\nabla^\eta,\nabla^D). \]
Thus, in this case, the definition of Weyl-harmonic depends only on the conformal class of $\eta$.
We are also interested in the case where $f$ is weakly conformal.
\begin{defi}
A map $f:\Sigma\to M$ is \emph{weakly conformal} if it is conformal whenever $df\neq 0$.
\end{defi}
If $z=x+iy$ is a complex coordinate on $\Sigma$ so that $\eta=e^{2\lambda}(dx^2+dy^2)$, then the equations
\[ \eta(\partial_x,\partial_y) = 0 \qquad \text{ and } \qquad  \eta(\partial_x,\partial_x) = \eta(\partial_y,\partial_y) \]
are conformally invariant.  If $f$ is weakly conformal then this implies that
\[ c( f_x,f_y ) = 0 \qquad \text{ and } \qquad  c( f_x,f_x ) =  c( f_y,f_y ). \]
Extending the conformal inner product to be complex bilinear, these are equivalent to the equation
\begin{equation}  c( f_z,f_z ) =0, \end{equation}
where $f_z=\frac{1}{2}(f_x-if_y)$.
Any point where $df$ is not full rank is called a singular point.
A branch point $p$ is a singular point where in some neighborhood of $p$, $f_z=z^k Z$ and $Z_p\neq 0$.

\begin{prop}
If $f:\Sigma\to M$ is weakly conformal, Weyl-harmonic, and non-constant, then $df$ is rank 2 except at an isolated set of branch points.
\end{prop}
\begin{proof}
Like the harmonic map equation, the Weyl-harmonic map equation can be written as the Laplace equation plus terms quadratic in $df$.
Thus the hypotheses of Aronzajn's unique continuation theorem and the Hartman-Wintner theorem \cite{mcduff2012} are still satisfied in the Weyl-harmonic case.
As $f$ is non-constant, Aronzajn's theorem implies that $df$ is not zero on an open set.
By the Hartman-Wintner Theorem, at every point there is an integer $m\ge 1$ so that in an isothermal coordinate chart,
\[ f(x,y) = h(x,y) + o(|(x,y)|^m) \quad \text{and} \quad df(x,y) = dh(x,y) + o(|(x,y)|^{m-1}). \]
for some non-zero homogeneous degree $m$ polynomial $h$.
The zeros of $dh$ are isolated, so the zeros of $df$ must be as well.
In fact, $h$ must also be weakly conformal and harmonic, thus the only zeros of $dh$ are branch point singularities.
Details and further analysis of the structure of these branch points is contained in \cite{micallef95}.
\end{proof}

If $f:\Sigma\to M$ is an immersion with only branch point singularities, then it is called a \emph{branched immersion}.
When $f$ is a branched immersion the conformal class $f^*c$ on $\Sigma$ can be defined across the branch points and $f$ is weakly conformal when $f^*c=[\eta]$.
A branched immersion which is Weyl-minimal away from the branch points is called a \emph{branched Weyl-minimal immersion}.
\begin{prop}
If $f:\Sigma\to M$ is weakly conformal then $f$ is Weyl-harmonic if and only if it is a branched Weyl-minimal immersion.
\end{prop}
\begin{proof}
In this case, away from the branch points,
\[ \tau(\eta,\nabla^\eta,\nabla^D)=e^{-2\lambda}(|f_x|^2_{g_\mu}+|f_y|^2_{g_\mu})\left(\mathbf{H}_{g_\mu}-\left(\alpha_\mu^{\sharp_{g_\mu}}\right)^\perp\right). \]
Comparing this with equation (\ref{weylH}) we see that $\tau(\eta,\nabla^\eta,\nabla^D)=0$ if and only if $\mathbf{H}^D=0$.
\end{proof}

There is a splitting $TM=T_f\Sigma\oplus N_f\Sigma$, where $T_f\Sigma=df(T\Sigma)$ away from the branch points.
At a branch point $p$, $f_z=z^k Z$, $Z_p\neq 0$, and $T_f\Sigma=\mathrm{span}\{\mathrm{Re}(Z_p),\mathrm{Im}(Z_p)\}$.
In both cases $N_f\Sigma$ is the orthogonal complement of $T_f\Sigma$.
\begin{defi}
For $M$ oriented, the \emph{twistor lifts} of $f$, $\tilde f_\pm:\Sigma\to\mathcal{Z}_\pm$ are determined by two complex structures on $f^*TM$.
There are two orthogonal complex structures  $J_\pm$ which agree with the complex structure of $\Sigma$ on $T_f\Sigma$.
The complex structure $J_+$ preserves the orientation while $J_-$ is orientation reversing.
The corresponding weightless K\"ahler form is then 
\[ \tilde f_\pm = c(J_\pm\cdot,\cdot) \]
\end{defi}
The complex structures $J_\pm$ determine a splitting of $f^*TM\otimes\mathbb{C}=f^*T^{(1,0)}_\pm M\oplus f^*T^{(0,1)}_\pm M$.
\end{subsubsection}
\end{subsection}
\begin{subsection}{Conformally Almost-Hermitian Manifolds}
A conformally almost-Hermitian manifold $(M^4,c,J)$ is an almost-complex manifold with a conformal structure satisfying
\[ c(X,Y)=c(JX,JY). \]
The conformal K\"ahler form $\omega_c=c(J\cdot,\cdot)$ can be viewed as a 2-form with values in $L^2$.  
Then there is a unique Weyl derivative satisfying $d^D\omega_c = 0$.
Fixing $\mu\in\Gamma(L)$ we can define this Weyl derivative using the \emph{Lee form}
\[ \theta_\mu = J\delta_{g_\mu}\omega_\mu=-(\delta_{g_\mu}\omega_\mu)J, \]
with $\delta_{g_\mu}$ denoting the divergence and $\omega_{c}=\mu^2\omega_\mu$.
In terms of an orthonormal coframe $\{e^i\}$ of $g_\mu$ satisfying
\[\omega_\mu = e^1\wedge e^2 + e^3\wedge e^4, \]
and $a_i=\langle e_i, \delta_{g_\mu}\omega_\mu \rangle_{g_\mu}$ one can check that
\[ d\omega_\mu = a_1 e^2\wedge e^3\wedge e^4-a_2 e^1\wedge e^3\wedge e^4+a_3 e^1\wedge e^2\wedge e^4-a_4 e^1\wedge e^2\wedge e^3. \]
Then the Lee form is
\[ \theta_\mu=-J\star d\omega_\mu = J (a_1 e^1+a_2 e^2+a_3 e^3+a_4 e^4) = a_1 e^2-a_2 e^1+a_3 e^4-a_4 e^3,\]
furthermore
\[ \theta_\mu\wedge\omega_\mu = d\omega_\mu. \]
Then for Weyl derivative $D=d+\alpha_\mu$,
\[ d^D\omega = d^D \mu^2\omega_\mu = 2\mu D(\mu)\omega+\mu^2 d\omega=2\mu^2\alpha_\mu\wedge\omega_\mu+\mu^2\theta_\mu\wedge\omega_\mu.\]
The canonical Weyl derivative of $(M,c,J)$ is then determined by setting $\alpha_\mu=-\frac{1}{2}\theta_\mu$.
The induced connection on $TM$ is given by
\[ \nabla^D_X Y = \nabla^{g_\mu}_X Y -\frac{1}{2} \theta_\mu(X)Y-\frac{1}{2}\theta_\mu(Y)X+\frac{1}{2}\langle X,Y\rangle_{g_\mu}\theta_\mu^{\sharp_{g_\mu}}. \]
The \textit{Nijenhuis Tensor} of $J$ is given by 
\[ N(X,Y)=[X,Y]+J[JX,Y]+J[X,JY]-[JX,JY]. \]
This is $J$ antilinear in both slots, that is $N(JX,Y)=-JN(X,Y)=N(X,JY)$.
For any vector field, $X$, $\nabla^D_X J$ is also $J$ antilinear.
\begin{prop}\label{dwJN} For any almost-Hermitian manifold
\[ \langle N(X,Y), JZ \rangle_{g_\mu} = d\omega_\mu(X,Y,Z)-d\omega_\mu(JX,JY,Z)-2\langle (\nabla^{g_\mu}_Z J)X,Y \rangle_{g_\mu}. \]
\end{prop}
This is proposition 4.2 in \cite{kobayashi1969} with different conventions.
The corresponding formula for conformally almost-Hermitian manifolds with a Weyl derivative is
\[ c(N(X,Y), JZ) = d^D\omega_c(X,Y,Z)-d^D\omega_c(JX,JY,Z)-2c( (\nabla^D_Z J)X,Y ). \]
\begin{prop}\label{weylJN} For any conformally almost-Hermitian manifold with canonical Weyl derivative $D$,
\[ c(N(X,Y), JZ) = -2c( (\nabla^D_Z J)X,Y ). \]
\end{prop}
\begin{coro}
The global section $s:(M,J)\to (\mathcal{Z}_+,\mathcal{J}_+)$ defined by $\omega_c$ is holomorphic.
\end{coro}
\begin{proof}
Letting $X,Y,Z\in TM$, by the proposition,
\[ (\nabla^D_Z\omega_c)(X,Y)=-\frac{1}{2}c(N(X,Y), JZ), \]
and by the symmetries of the Nijenhuis tensor,
\[ (\nabla^D_{JZ}\omega_c)(X,Y)=\frac{1}{2}c(N(X,Y), Z)=\frac{1}{2}c(JN(X,Y), JZ)=-\frac{1}{2}c(N(JX,Y), JZ). \]
Therefore $(\nabla^D_{JZ}\omega_c)=\mathcal{J}_+(\nabla^D_Z\omega_c)$, and by equations (\ref{vhConn}), (\ref{calJh}), and (\ref{calJv})
\[ ds(JZ) = (JZ)^h + (\nabla^D_{JZ}\omega_c)^v=\mathcal{J}_+Z^h+\mathcal{J}_+(\nabla^D_Z\omega_c)^v=\mathcal{J}_+ds(Z). \]
\end{proof}

\end{subsection}
\end{section}

\begin{section}{Main Theorems}

\begin{subsection}{Twistor Correspondence}

Following \cite{eells1985} we now show there is a correspondence between weakly conformal Weyl-harmonic maps and non-vertical $\mathcal{J}_\pm$-holomorphic maps into the weightless twistor space with complex structure given by (\ref{calJh}) and (\ref{calJv}).

The twistor lifts of a weakly conformal map $f:\Sigma\to M$ are given by 
\[ \tilde f_\pm = \left(\frac{(1\pm \star)f_z \wedge f_{\bar z}}{i c(f_z,f_{\bar z})}\right)^{\flat_c}, \]
where $i=\sqrt{-1}$. The natural isomorphism $\flat_c:L^{-1}TM\to LT^*M$ preserves weights, but interchanges the holomorphic and anti-holomorphic spaces.
\begin{thm}
For any Weyl manifold $(M,c,D)$, a weakly conformal map $f:\Sigma\to M$ is Weyl-harmonic if and only if the twistor lifts $\tilde f_\pm:\Sigma\to \mathcal{Z}_\pm$ are $\mathcal{J}_\pm$-holomorphic.
\end{thm}
\begin{proof}
The twistor lifts $\tilde f_\pm$ are $\mathcal{J}_\pm$-holomorphic provided $d\tilde f_\pm (\partial_z) \in T^{(1,0)}\mathcal{Z}_\pm$.
We have
\[ d\tilde f_\pm (\partial_z) = (f_z)^h + \left(\nabla^D_{\partial_z}\tilde f_\pm\right)^v, \]
and since $f_z\in T^{(1,0)}_\pm M$, we have $(f_z)^h \in H(T^{(1,0)}\mathcal{Z}_\pm)$.
Thus all that is required is $\nabla^D_{\partial_z}\tilde f_\pm \in L^2\Lambda^{(2,0)}T^*M$.
We find that
\begin{align*}
(\nabla^D_{\partial_z}\tilde f_\pm)^{(0,2)}_\pm & = \left(\frac{\sqrt{2}f_z \wedge (\nabla^D_{\partial_z}f_{\bar z})^{(1,0)}_\pm}{i c(f_z,f_{\bar z})}\right)^{\flat_c}, \\
(\nabla^D_{\partial_z}\tilde f_\pm)^{(1,1)}_\pm & = 0, \\
(\nabla^D_{\partial_z}\tilde f_\pm)^{(2,0)}_\pm & = \left(\frac{\sqrt{2}(\nabla^D_{\partial_z}f_z)^{(0,1)}_\pm \wedge f_{\bar z}}{i c(f_z,f_{\bar z})}\right)^{\flat_c}. 
\end{align*}
It follows that $\tilde{f}_\pm$ is pseudo-holomorphic map if and only if $(\nabla^D_{\partial_z}f_{\bar z})^{(1,0)}_\pm=k f_z$, for some function $k$.
Taking the conformal inner-product with $f_{\bar z}$ gives
\[ c((\nabla^D_{\partial_z}f_{\bar z})^{(1,0)}_\pm,f_{\bar z})=k c(f_z,f_{\bar z}). \]
Since $f_{\bar z}\in T^{(0,1)}_\pm M$ this is just
\[ c(\nabla^D_{\partial_z}f_{\bar z},f_{\bar z})=k c(f_z,f_{\bar z}),\]
and since $c$ is $\nabla^D$ parallel we have
\[ \nabla^D_{\partial_z}c(f_{\bar z},f_{\bar z})=2k c(f_z,f_{\bar z}).\]
For $f$ weakly conformal, this shows that $k=0$.  Therefore $\tilde f_\pm$ is $\mathcal{J}$-holomorphic if and only if 
$(\nabla^D_{\partial_z}f_{\bar z})^{(1,0)}_\pm=0$, but since $\nabla^D_{\partial_z}f_{\bar z}$ is real, this can only be true when it is zero.
\end{proof}
\begin{coro}
There is a one-to-one correspondence between weakly conformal Weyl-harmonic maps to $(M,c,D)$ and non-vertical $\mathcal{J}_\pm$-holomorphic maps to the twistor space.
\end{coro}
\begin{proof}
It only remains to show that for a non-vertical $\mathcal{J}$-holomorphic curve, $\phi:\Sigma\to \mathcal{Z}_\pm$ 
the projection $\bar{\phi}:\Sigma\to M$ is weakly conformal and Weyl-harmonic.
It is clearly weakly conformal as $\bar\phi_z$ is holomorphic with respect to the complex structure defined by $\phi$, which implies $c(\bar\phi_z,\bar\phi_z)=0$.
It is Weyl-harmonic as $\phi$ is its twistor lift and is $\mathcal{J}$-holomorphic.
\end{proof}
\begin{coro}
The $J$-holomorphic curves $f:\Sigma \to M$ are weakly conformal and Weyl-harmonic. 
\end{coro}
\begin{proof}
The composition with the section $s:M\to\mathcal{Z}_+$ determined by $J$ is a $\mathcal{J}_+$-holomorphic curve of $\mathcal{Z}_+$.
\end{proof}

\end{subsection}

\begin{subsection}{Adjunction Inequality}
In this section we prove theorem \ref{webforthm}
\begin{proof} Fix a metric $\langle \, , \rangle$ in the conformal class.
For a holomorphic normal coordinate $z$ on $\Sigma$, split $f_z$ into its holomorphic and antiholomorphic parts
\[ \alpha = \frac{1}{2}(f_z-iJf_z) \qquad \bar\beta = \frac{1}{2}(f_z+iJf_z), \]
we have 
\[\langle \alpha, \alpha \rangle = 0 = \langle \beta,\beta \rangle. \]
When $f$ is weakly conformal, this implies that 
\[ \langle \alpha, \bar\beta \rangle = 0. \]
Thus $\alpha$ and $\beta$ are Hermitian orthogonal and away from their zeros span the holomorphic tangent bundle $f^*T^{(1,0)}M$.
The Weyl-harmonic map equation in coordinates is
\[ \nabla^D_{\partial_{\bar z}} f_z=0. \]
Since $\nabla^D$ does not preserve the almost-complex structure, we write the equation using the connection $\nabla^{D,J}$, given by
\[ \nabla^{D,J}_X Y = \nabla^D_X Y -\frac{1}{2}J(\nabla^D_X J)Y. \]
This connection preserves the complex structure, and thus preserves the holomorphic and anti-holomorphic tangent spaces.
In terms of this connection the Weyl-harmonic map equation is 
\[ \nabla^{D,J}_{\partial_{\bar{z}}} \frac{\partial\phi}{\partial z}=-\frac{1}{2}J(\nabla^D_{\partial_{\bar{z}}} J) \frac{\partial\phi}{\partial z}. \]
Since $\nabla^D_{\partial_{\bar{z}}} J$ is $J$ anti-linear, it maps from $T^{(1,0)}M$ to $T^{(0,1)}M$ and from $T^{(0,1)}M$ to $T^{(1,0)}M$.
The Weyl-harmonic map equation can  then be written in terms of $\alpha$ and $\bar\beta$ as
\begin{align}
\nabla^{D,J}_{\partial_{\bar{z}}} \alpha&=-\frac{i}{2}(\nabla^D_{\partial_{\bar{z}}} J) \bar\beta,\\
\nabla^{D,J}_{\partial_{\bar{z}}} \bar\beta&=\frac{i}{2}(\nabla^D_{\partial_{\bar{z}}} J) \alpha.
\end{align}Using proposition \ref{weylJN}, a weakly conformal Weyl-harmonic map must satisfy
\begin{align*}\left\langle \nabla^{D,J}_{\partial_{\bar{z}}} \alpha, \bar\alpha \right\rangle &= -\frac{i}{2}\left\langle (\nabla^D_{\partial_{\bar{z}}} J)\bar\beta, \bar\alpha \right\rangle, & \left\langle \nabla^{D,J}_{\partial_{\bar{z}}} \alpha, \bar\beta \right\rangle &= -\frac{i}{2}\left\langle (\nabla^D_{\partial_{\bar{z}}} J)\bar\beta, \bar\beta \right\rangle, \\
&= \frac{i}{4}\left\langle  N\left(\bar\beta,\bar\alpha\right),Jf_{\bar z} \right\rangle, &&= \frac{i}{4}\left\langle  N\left(\bar\beta,\bar\beta\right),Jf_{\bar z} \right\rangle,\\
&= \frac{1}{4}\langle N(\bar\alpha,\bar\beta),\bar\alpha \rangle, &&=0,
\end{align*}
where the last line follows from $\bar\alpha,\bar\beta \in f^* T^{(0,1)}M$, which implies $N(\bar\alpha,\bar\beta)\in f^*T^{(1,0)}M$.
This implies that away from the zeros of $\alpha$ and $\beta$,
\begin{equation} \nabla^{D,J}_{\partial_{\bar{z}}} \alpha = \frac{\langle N(\bar\alpha,\bar\beta),\bar\alpha \rangle}{4\|\alpha\|^2}\alpha. \end{equation}
Similarly
\begin{align*}\left\langle \nabla^{D,J}_{\partial_{\bar{z}}} \bar\beta, \beta \right\rangle &= \frac{i}{2}\left\langle (\nabla^D_{\partial_{\bar{z}}} J)\alpha, \beta \right\rangle,
& \left\langle \nabla^{D,J}_{\partial_{\bar{z}}} \bar\beta, \alpha \right\rangle &= \frac{i}{2}\left\langle (\nabla^D_{\partial_{\bar{z}}} J)\alpha, \alpha \right\rangle, \\
&= -\frac{i}{4}\left\langle  N\left(\alpha,\beta\right),Jf_{\bar z} \right\rangle, & &= -\frac{i}{4}\left\langle  N\left(\alpha,\alpha\right),Jf_{\bar z} \right\rangle, \\
&= \frac{1}{4}\langle N(\alpha,\beta),\beta \rangle, &&=0.
\end{align*}
This gives
\begin{equation} \nabla^{D,J}_{\partial_{\bar{z}}} \bar\beta = \frac{\langle N(\alpha,\beta),\beta \rangle}{4\|\beta\|^2}\bar\beta. \end{equation}
By the Koszul-Malgrange theorem \cite{Koszul1958}, there are holomorphic structures on $f^*T^{(1,0)}M$ and $f^*T^{(0,1)}M$ so that 
\[ \bar\partial X = \nabla^{D,J}_{\partial_{\bar z}} X \otimes d\bar z. \]
Then for a Weyl-harmonic map 
\begin{align}
\bar\partial \alpha &= \frac{\langle N(\bar\alpha,\bar\beta),\bar\alpha \rangle}{4\|\alpha\|^2}\alpha\otimes d\bar z,\\
\bar\partial \bar\beta &= \frac{\langle N(\alpha,\beta),\beta \rangle}{4\|\beta\|^2}\bar\beta\otimes d\bar z.
\end{align}
The Bers-Vekua similarity principle (see \cite{jost1991}) implies that near any point $p\in\Sigma$ we have
\[ \alpha = \gamma_p e^{\sigma_p}, \qquad \bar\beta = \delta_p e^{\tau_p}, \]
for some local holomorphic sections $\gamma_p$ of $f^*T^{(1,0)}M$, $\delta_p$ of $f^*T^{(0,1)}M$, and some bounded functions $\sigma_p, \tau_p$.
This can be used to define the indices
\begin{align}
R &= \sum_{f_z(p)=0} \ord_p(f_z) \ge 0,\\
Q &= \sum_{\alpha(p)=0} \ord_p(\gamma_p) - R \ge 0,\\
P &= \sum_{\bar\beta(p)=0} \ord_p(\delta_p) - R \ge 0.
\end{align}
These are the total ramification index $R$, the number of anti-complex points $Q$, and the number of complex points $P$.
Following \cite{eells1985}, these determine the degrees of the line bundles spanned by the vector valued one forms $f_z dz$, $\alpha dz$ and $\bar\beta dz$ respectively.
If $[f_z]$, $[\alpha]$, and $[\beta]$ are line bundles generated by the locally defined sections then we have
\begin{align*}
R &= -\chi(\Sigma)+c_1([f_z]),\\
Q+R &= -\chi(\Sigma)+c_1([\alpha]),\\
P+R &= -\chi(\Sigma)-c_1([\beta]).
\end{align*}
We also have $f^*T^{(1,0)}M = [\alpha]\oplus [\beta]$, and since $\alpha$ and $\bar\beta$ span a negatively oriented, maximal isotropic subspace of $f^*TM\otimes \mathbb{C}$ which contains $f_z$, it must be that $f^*T^{(1,0)}_- M=[\alpha]\oplus [\bar\beta]$.
Therefore we have
\begin{align*}
c_1(f^*T^{(1,0)}M) &= Q-P, \\
c_1(f^*T^{(1,0)}_-M) &= Q+P+2R+2\chi(\Sigma), \\
 &= Q + P + 2c_1([f_z]), \\
 &= Q + P + 2\chi(T_f\Sigma).
\end{align*}
Since $c_1(f^*T^{(1,0)}_-M) = \chi(T_f\Sigma)-\chi(N_f\Sigma)$ we now have the Webster's formulas
\begin{align}
c_1(f^*T^{(1,0)}M) &= Q-P, \\
\chi(T_f\Sigma)+\chi(N_f\Sigma) &= -P-Q.
\end{align}
\end{proof}
Since $P$ and $Q$ are both non-negative, this implies the adjunction inequality (\ref{adj}) of corollary \ref{adjcor}
\begin{align}
\chi(T_f\Sigma)+\chi(N_f\Sigma)+c_1(f^*T^{(1,0)}M) &= -2P \le 0,\\
\chi(T_f\Sigma)+\chi(N_f\Sigma)-c_1(f^*T^{(1,0)}M) &= -2Q \le 0.
\end{align}
\end{subsection}

\end{section}

\begin{section}{Examples}

\begin{subsection}{Hopf Surfaces}
The primary Hopf surface $M=S^1\times S^3$ is fibered over $S^2$ with fiber $T^2$. 
The bundle projection is just the projection to the $S^3$ component followed by the Hopf map.
There is a Hermitian structure on $M$ induced by the standard Hermitian structures on the base and fiber.
The Lee form is just $d\phi$, where $\phi$ is the angle along $S^1$.
Every fiber is $J$-holomorphic and is therefore Weyl-minimal.
It is also minimal as $\theta^\sharp$ is tangent to the fiber.

In addition, there is a Lagrangian Weyl-minimal surface for every great circle $\gamma$ in the base $S^2$.
To see this consider the Clifford torus in $S^3$ which maps to $\gamma$ under the Hopf map.
This torus contains two great circles on $S^3$, one tangent to the fiber and one perpendicular to the fiber.
The great circle perpendicular to the fiber times the product $S^1$ gives a Lagrangian totally geodesic $T^2$ to which $\theta^\sharp$  is tangent, and is therefore Weyl-minimal.

Since $\theta=d\phi$ is closed we can look at the universal cover $\tilde M=\mathbb{R}\times S^3$.
Using $\phi$ as the coordinate on $\mathbb{R}$ the metric is just
\[ g_{\tilde M} = d\phi^2 + g_{S^3}.  \]
Therefore the Weyl-minimal surfaces will lift to minimal surfaces of the conformal metric
\[ e^{2\phi}g_{\tilde M} = e^{2\phi}d\phi^2+e^{2\phi}g_{S^3} = (de^\phi)^2+e^{2\phi}g_{S^3}.\]
Using the new coordinate $r=e^\phi$ this is just the (incomplete) flat metric on $\mathbb{R}^4\setminus 0 \cong \tilde M$.
\[ dr^2 + r^2 g_{S^3}. \]
Any surface lifted from $M$ will be invariant under deck transformation $\phi\mapsto\phi+2\pi$ or $r\mapsto e^{2\pi}r$.
The Weyl-minimal surfaces described above correspond to the planes through the origin in $\mathbb{R}^4$.
\end{subsection}
\begin{subsection}{$U(1)\times U(1)$ Principal Bundles over a Riemann Surface}
Let $p:M\to\Sigma$ be a $U(1)\times U(1)$ principal bundle over a Riemann surface $\Sigma$ with volume form $\omega_\Sigma$.
If $i\beta$ is a connection form then $\beta$ is an $\mathbb{R}^2$-valued form with components $\beta_1$ and $\beta_2$.
If $\tilde\omega_\Sigma=p^*\omega_\Sigma$ then $d\beta=F\tilde\omega_\Sigma$ where $F=(F_1,F_2):\Sigma\to\mathbb{R}^2$.
The K\"ahler form $\omega=\beta_1\wedge\beta_2+\tilde\omega_\Sigma$ has exterior derivative 
\[ d\omega = (F_1\beta_2-F_2\beta_1)\wedge\tilde\omega_\Sigma = (F_1\beta_2-F_2\beta_1)\wedge\omega.\]
Therefore the Lee form is $\theta=F_1\beta_2-F_2\beta_1$.  For a constant curvature connection, this will be closed.
The Hopf surface is a special case for this example where $\Sigma=S^2$ with the round metric, and $M$ has associated bundle $M\times\mathbb{C}^2/U(1)\times U(1)=\mathbb{C}\oplus\mathbf{K}$.
As in that case, the fiber is always a $J$-holomorphic curve and therefore Weyl-minimal.
If a closed geodesic on $\gamma:S^1\to\Sigma$ has a closed horizontal lift $\tilde\gamma$ and the connection has constant curvature then 
$\tilde\gamma(s) \cdot (e^{-iF_2 t},e^{iF_1 t})$ parametrizes a Lagrangian minimal torus on $M$ to which $\theta^\sharp$ is tangent, and thus Weyl-minimal.
\end{subsection}
\end{section}

\bibliography{mybib}
\bibliographystyle{plain}

\end{document}